\documentclass[12pt]{amsart}
\usepackage{fullpage}
\usepackage{amsfonts,amscd}
\usepackage{amssymb}
\usepackage{url}
\usepackage{graphicx}
\usepackage[english]{babel}
\theoremstyle{plain}
\newtheorem{theorem}                {Theorem}      [section]

\theoremstyle{definition}

\newtheorem{remark}       [theorem]  {Remark}

\numberwithin{equation}{section}

\def \R{{\mathbb R}}
\def \s{{\mathbb S}}

\usepackage{color}

\DeclareMathOperator{\trace}{trace}
\numberwithin{equation}{section}

\begin{document}

\title[]{Biharmonic submanifolds into ellipsoids}

\author{S.~Montaldo}
\address{Universit\`a degli Studi di Cagliari\\
Dipartimento di Matematica e Informatica\\
Via Ospedale 72\\
09124 Cagliari, Italia}
\email{montaldo@unica.it}

\author{A.~Ratto}
\address{Universit\`a degli Studi di Cagliari\\
Dipartimento di Matematica e Informatica\\
Viale Merello 93\\
09123 Cagliari, Italia}
\email{rattoa@unica.it}

\begin{abstract}
In this paper we construct proper biharmonic submanifolds into various types of ellipsoids. We also prove, in this context, some useful composition properties which can be used to produce large families of new proper biharmonic immersions.
\end{abstract}



\thanks{Work supported by P.R.I.N. 2010/11 -- Variet\`a reali e complesse: geometria, topologia e analisi armonica -- Italy.}

\maketitle

\section{Introduction}\label{intro}

{\it Harmonic maps}  are critical points of the {\em energy} functional
\begin{equation}\label{energia}
E(\varphi)=\frac{1}{2}\int_{M}\,|d\varphi|^2\,dv_g \,\, ,
\end{equation}
where $\varphi:(M,g)\to(N,h)$ is a smooth map between two Riemannian
manifolds $M$ and $N$. In analytical terms, the condition of harmonicity is equivalent to the fact that the map $\varphi$ is a solution of the Euler-Lagrange equation associated to the energy functional \eqref{energia}, i.e.
\begin{equation}\label{harmonicityequation}
    {\trace} \, \nabla d \varphi =0 \,\, .
\end{equation}
The left member of \eqref{harmonicityequation} is a vector field along the map $\varphi$, or, equivalently, a section of the pull-back bundle $\varphi^{-1} \, (TN)$: it is called {\em tension field} and denoted $\tau (\varphi)$.

A related topic of growing interest deals with the study of the so-called {\it biharmonic maps}: these maps, which provide a natural generalisation of harmonic maps, are the critical points of the bienergy functional (as suggested by Eells--Lemaire \cite{EL83})
\begin{equation*}\label{bienergia}
    E_2(\varphi)=\frac{1}{2}\int_{M}\,|\tau (\varphi)|^2\,dv_g \,\, .
\end{equation*}
In \cite{Jiang} G.~Jiang derived the first variation and the second variation formulas for the bienergy. In particular, he showed that the Euler-Lagrange equation associated to $E_2(\varphi)$ is
\begin{equation}\label{bitensionfield}
    \tau_2(\varphi) = - J\left (\tau(\varphi) \right ) = - \triangle \tau(\varphi)- \trace R^N(d \varphi, \tau(\varphi)) d \varphi = 0 \,\, ,
    \end{equation}
where $J$ denotes (formally) the Jacobi operator of $\varphi$, $\triangle$ is the rough Laplacian on sections of $\varphi^{-1} \, (TN)$ that, for a local orthonormal frame $\{e_i\}_{i=1}^m$ on $M$, is defined by
\begin{equation}\label{roughlaplacian}
    \Delta=-\sum_{i=1}^m\{\nabla^{\varphi}_{e_i}
    \nabla^{\varphi}_{e_i}-\nabla^{\varphi}_
    {\nabla^{M}_{e_i}e_i}\}\,\,,
\end{equation}
and
\begin{equation*}\label{curvatura}
    R^N (X,Y)= \nabla_X \nabla_Y - \nabla_Y \nabla_X -\nabla_{[X,Y]}
\end{equation*}
is the curvature operator on $(N,h)$.
We point out that \eqref{bitensionfield} is a {\it fourth order}  semi-linear elliptic system of differential equations. We also note that any harmonic map is an absolute minimum of the bienergy, and so it is trivially biharmonic. Therefore, a general working plan is to study the existence of biharmonic maps which are not harmonic:  these shall be referred to as {\it proper biharmonic maps}. We refer to \cite{SMCO} for existence results and general properties of biharmonic maps.\\

An immersed sub\-mani\-fold into a Riemannian manifold $(N,h)$ is called a {\it biharmonic sub\-mani\-fold} if the immersion is a biharmonic map.
In a purely geometric context, B.-Y.~Chen \cite{Chen} defined biharmonic submanifolds $M \,\subset\, \R^n$ of the Euclidean space as those with harmonic mean curvature vector field, that is $\Delta \, H=(\Delta H_1,\ldots,\Delta H_n)=0$, where $H=(H_1,\ldots,H_n)$ is the mean curvature vector as seen in $\R^n$ and $\Delta$ is the Beltrami-Laplace operator on $M$. It is important to point out that, if we apply the definition of biharmonic maps to immersions into the Euclidean space, we recover Chen's notion of biharmonic submanifolds. In this sense, our work can be regarded in the spirit of a generalization of Chen's biharmonic submanifolds.\\

A general result of Jiang \cite{Jiang} tells us that a compact, orientable, biharmonic submanifold $M$ into a manifold $N$ such that ${\rm Riem}^N \, \leq \,0$ is necessarily minimal. Moreover, C.~Oniciuc, in \cite{O02}, proved that also CMC biharmonic isometric immersions into a manifold $N$  with  ${\rm Riem}^N \, \leq \,0$ are necessarily minimal. In fact, it is still open the Chen's conjecture: biharmonic submanifolds into a non-positive constant sectional curvature manifold are minimal. The Chen's conjecture was generalized in \cite{CMO01} for biharmonic submanifolds into a Riemannian manifold with non-positive sectional curvature, although Y.~Ou and L.~Tang found in \cite{OT} a counterexample. These facts have pushed research towards the investigation of biharmonic submanifolds of the Euclidean sphere (see \cite{BMO13,BMO10,BMO08,BO11,BO09,CMO01,CMO02} for an overview of the main results in this context). A further step is the study of biharmonic submanifolds into Euclidean ellipsoids, because these manifolds are geometrically rich and interestingly do not have constant sectional curvature: in \cite{Mont_Ratto} we obtained a complete classification of proper biharmonic curves into 3-dimensional ellipsoids and, more generally, into any non-degenerate quadric. In this paper, we shall focus on proper biharmonic submanifolds of dimension $\geq 2$. \\

\section{Biharmonic submanifolds into ellipsoids}


 We begin with the study of biharmonic submanifolds into Euclidean ellipsoids $Q^{p+q+1}(c,d)$ defined as follows:
\begin{equation*}\label{def-ellissoide}
    Q^{p+q+1}(c,d)=\left \{ (x,y)\in \, \R^{p+1} \times \R^{q+1}= \R^n \,\, : \,\, \frac{|x|^2}{c^2}+ \frac{|y|^2}{d^2} =1 \right \} \,\,,
\end{equation*}
where $c,\,d$ are fixed positive constants. The symmetry of $Q^{p+q+1}(c,d)$ makes it natural to look for biharmonic generalized Clifford's tori. More precisely, we shall study isometric immersions of the following type:
\begin{equation}\label{CliffordToriimmersion}
\left .
  \begin{array}{cccccc}
    i\,\,:\,\,& S^p(a) &\times& S^q(b)&\longrightarrow &Q^{p+q+1}(c,d)\\
    &&&&& \\
    &(x_1,\ldots,x_{p+1}&,&y_1,\ldots,y_{q+1})& \longmapsto & (x_1,\ldots,x_{p+1},y_1,\ldots,y_{q+1}) \,\, , \\
  \end{array}
\right .
\end{equation}
where $i$ denotes the inclusion and the radii $a,b$ must satisfy the following condition:
\begin{equation}\label{condizioneimmersionetoroinellissoide}
    \frac{a^2}{c^2}+  \frac{b^2}{d^2}= 1 \,\, .
\end{equation}
In this context, we shall prove the following result:
\begin{theorem}\label{teorema1} Let $i\,:\,S^p(a) \times S^q(b) \to Q^{p+q+1}(c,d)$ be an isometric immersion as in \eqref{CliffordToriimmersion}. If
\begin{equation}\label{condizioneminimalita}
     a^2 = c^2 \, \frac{p}{p+q} \,\, ; \quad b^2 = d^2 \,
     \frac{q}{p+q}  \,\,
\end{equation}
then the immersion is minimal. If \eqref{condizioneminimalita} does not hold and
\begin{equation}\label{condizioneproperbiharmonic}
    a^2 = c^2 \, \frac{c}{c+d} \,\, ; \quad b^2 = d^2 \, \frac{d}{c+d} \,\, ,
\end{equation}
then the immersion is proper biharmonic.
\end{theorem}
\begin{remark} We observe that, interestingly, if $c=p$ and $d=q$, then we have generalized minimal Clifford's tori, but we do not have proper biharmonic submanifolds of the type \eqref{CliffordToriimmersion}.

We also point out that, according to Theorem \ref{teorema1}, the ellipsoid $Q^{3}(c,d)$ ($p=q=1, \, c\neq d$) admits a proper biharmonic torus, while in $S^3$ there exists no genus $1$ proper biharmonic submanifold (see \cite{CMO01}).
\end{remark}
\begin{proof} We shall work essentially by using coordinates in $\R^n$, suitably restricted to the ellipsoid or to the torus, according to necessity. In particular, the splitting $$(x,y)=(x_1,\ldots,x_{p+1},y_1,\ldots,y_{q+1})$$
will be used in an obvious way, without further comments. The symbol $\langle\, ,\,\rangle  $ will denote the Euclidean scalar product (whether in $\R^{p+1}$, $\R^{q+1}$ or $\R^{n}$ will be clear from the context). We shall use a superscript $Q$ for objects concerning the ellipsoid, while the letter $T$ will appear for reference to the torus $T=S^p(a) \times S^q(b)$.

We shall need to know the algebraic conditions which ensure that a given vector field is tangent either to the torus or to the ellipsoid. More specifically, a vector field
$$
W=(X,Y) \,\,,
$$
where
$$
X=\sum_{i=1}^{p+1}\, X^i \, \frac{\partial}{\partial x_i} \quad {\rm and} \quad Y=\sum_{j=1}^{q+1}\, Y^j \, \frac{\partial}{\partial y_j} \,\, ,
$$
is tangent to $Q^{p+q+1}(c,d)$ if and only if
\begin{equation*}\label{condiztangenzaellissoide}
   \sum_{i=1}^{p+1}\, \frac{1}{c^2}\, x_i \, X^i \, + \, \sum_{j=1}^{q+1}\, \frac{1}{d^2}\, y_j \, Y^j \, = \, 0 \,\, .
\end{equation*}
In the same order of ideas,
$W$ is tangent to the torus $T$ if and only if
\begin{equation}\label{condiztangenzatoro}
   \sum_{i=1}^{p+1}\, \, x_i \, X^i \, = \,0 = \, \sum_{j=1}^{q+1}\,  y_j \, Y^j \,\, .
\end{equation}

To end preliminaries, we observe that the vector field
\begin{equation}\label{normaleunitariaellissoide}
    \eta^Q =\frac{\eta^Q_1}{|\eta^Q_1|} \,\, ,
\end{equation}
where
\begin{equation*}\label{normaleellissoide}
    \eta^Q_1 = \left( \frac{1}{c^2}\,x_1,\ldots,\,\frac{1}{c^2}\,x_{p+1},\,\frac{1}{d^2}\,y_1,\ldots,\,\frac{1}{d^2}\,y_{q+1} \right )\,\, ,
\end{equation*}
represents a unit normal vector field on the ellipsoid $Q^{p+q+1}(c,d)$. Note, for future use, that the equality
\begin{equation}\label{normacostantesutoro}
    |\eta^Q_1|^2= \frac{a^2}{c^4}+\frac{b^2}{d^4}
\end{equation}
holds on $T$. Similarly, the vector
\begin{equation}\label{normaleunitariatoro}
    \eta^T =\frac{\eta^T_1}{|\eta^T_1|} \,\, ,
\end{equation}
where
\begin{equation}\label{normaletoro}
    \eta^T_1 = \left( \frac{c^2}{a^2}\,x_1,\ldots,\,\frac{c^2}{a^2}\,x_{p+1},\,-\, \frac{d^2}{b^2}\,y_1,\ldots,\,-\, \frac{d^2}{b^2}\,y_{q+1} \right )\,\, ,
\end{equation}
represents a unit normal vector on the torus $T$ viewed as a submanifold of the ellipsoid $Q^{p+q+1}(c,d)$. We also note that
\begin{equation}\label{normacostantesutorobis}
    |\eta^T_1|^2= \frac{c^4}{a^2}+\frac{d^4}{b^2}
\end{equation}
on $T$.
In order to compute the tension and the bitension fields, it is convenient to make explicit the formulas which will enable us to calculate the relevant covariant derivatives. More precisely, following, for instance, \cite{DoCarmo}, we know that
\begin{equation}\label{derivatacovainQ}
\nabla_{W_1}^Q \, W_2=\nabla_{W_1}^{\R^n} \, W_2 - B^Q (W_1,W_2) \,\, ,
\end{equation}
where $B^Q (W_1,W_2)$ denotes the second fundamental form of the ellipsoid $ Q^{p+q+1}(c,d)$ into $\R^n$. We now need to do some work to make \eqref{derivatacovainQ} more explicit:
\begin{eqnarray}\label{calcolosecondaformaQ}
\nonumber
  B^Q (W_1,W_2) &=& -\, \langle \,\nabla_{W_1}^{\R^n} \, \eta^Q , \, W_2 \,\rangle    \, \eta^Q\\ \nonumber
   &=& -\, \langle W_1 \left ( \frac{1}{|\eta_1^Q|} \right ) \,\eta_1^Q +\,\frac{1}{|\eta_1^Q|}\, \nabla_{W_1}^{\R^n} \, \eta_1^Q , \, W_2 \,\rangle   \, \eta^Q  \\
   &=& -\, \langle \, \frac{1}{|\eta_1^Q|}\, \nabla_{W_1}^{\R^n} \, \eta_1^Q , \, W_2 \,\rangle   \, \eta^Q \,\, .
\end{eqnarray}
Next, we compute
\begin{equation}\label{derivatacovadietaQ}
\nabla_{W_1}^{\R^n} \, \eta_1^Q= \frac{1}{c^2}\,\sum_{i=1}^{p+1}\,   X_1^i \, \frac{\partial}{\partial x_i} + \,\frac{1}{d^2}\, \sum_{j=1}^{q+1}\, Y_1^j \,\frac{\partial}{\partial y_j} \,\, .
\end{equation}
Finally, using \eqref{derivatacovadietaQ} into \eqref{calcolosecondaformaQ}, we obtain:
\begin{equation}\label{espressionefinaleBinQ}
  B^Q (W_1,W_2)=  \, -\, \frac{1}{|\eta_1^Q|} \, \left [ \,\frac{1}{c^2}\,\langle X_1,X_2\rangle  + \frac{1}{d^2}\,\langle Y_1,Y_2\rangle  \right ]\, \eta^Q \,\, ,
\end{equation}
which in \eqref{derivatacovainQ} yields:
\begin{equation*}\label{derivatacovainQfinaleformula}
\nabla_{W_1}^Q \, W_2=\nabla_{W_1}^{\R^n} \, W_2 +\, \frac{1}{|\eta_1^Q|} \, \left [ \,\frac{1}{c^2}\,\langle X_1,X_2\rangle  + \frac{1}{d^2}\,\langle Y_1,Y_2\rangle  \right ]\, \eta^Q \,\, .
\end{equation*}
Now, we are in the right position to proceed to the computation of the tension field $\tau$ of our immersion \eqref{CliffordToriimmersion}. Indeed, by definition,
\begin{equation}\label{tensionfieldtoro}
    \tau= {\trace}\, B^T (\cdot, \, \cdot) \,\, ,
\end{equation}
where
\begin{eqnarray}\label{secondaformaT}
 \nonumber
     B^T (W_1, \, W_2)&=& \,-\,\langle \,\nabla_{W_1}^{Q} \, \eta^T , \, W_2 \,\rangle    \, \eta^T \\ \nonumber
     &=&-\,\frac{1}{|\eta_1^T|^2}\langle \,\nabla_{W_1}^{Q} \, \eta_1^T , \, W_2 \,\rangle    \, \eta_1^T  \\ \nonumber
     &=&-\,\frac{1}{|\eta_1^T|^2}\langle \,\nabla_{W_1}^{\R^n} \, \eta_1^T , \, W_2 \,\rangle    \, \eta_1^T  \\
     &=&-\,\frac{1}{|\eta_1^T|^2} \, \left [ \,\frac{c^2}{a^2}\,\langle X_1,X_2\rangle  - \,\frac{d^2}{b^2}\,\langle Y_1,Y_2\rangle  \right ]\, \, \eta_1^T \,\, .
\end{eqnarray}
Let now $X_i$ , $i=1,\ldots,p$ and $Y_j$ , $j=1,\ldots,q$, be local orthonormal bases of $S^p(a)$ and $S^q(b)$ respectively.
By using \eqref{secondaformaT} in \eqref{tensionfieldtoro} we find:
\begin{eqnarray}\label{tensionfieldtorocalcolato}
    \tau&= & \sum_{i=1}^p\, B^T \left ( (X_i,0),(X_i,0)\right ) +\,\sum_{i=1}^q\, B^T \left ( (0,Y_j),(0,Y_j)\right ) \\ \nonumber
    &= & -\, \frac{1}{|\eta_1^T|^2} \, \left [ \frac{p \, c^2}{a^2}-\,\frac{q \, d^2}{b^2} \right ] \, \eta_1^T \\ \nonumber
    &= &\lambda \,\eta_1^T \,\, ,
\end{eqnarray}
where, taking into account \eqref{normacostantesutorobis}, we have set
\begin{equation*}\label{valoredilambda}
    \lambda= \,- \, \left [ \frac{c^4}{a^2}+\frac{d^4}{b^2} \right ]^{-1}\, \left [ \frac{p \, c^2}{a^2}-\,\frac{q \, d^2}{b^2} \right ] \,\, .
\end{equation*}
In particular, using \eqref{condizioneimmersionetoroinellissoide}, it is now immediate to conclude that \eqref{condizioneminimalita} is equivalent to the minimality of the immersion. \\

Next, we proceed to the computation of the bitension field $\tau_2$. To this purpose, we must apply \eqref{bitensionfield} in the case that $\varphi=i$. We begin with the computation of $\Delta \tau$ . It is convenient to choose a \emph{geodesic} local orthonormal frame obtained from geodesic local orthonormal frames on each factor of the torus. Under this assumption the terms $\nabla^{T}_{e_i}\, e_i$ in the formula \eqref{roughlaplacian} vanish (note that this simplification is acceptable because we shall not need to compute covariant derivatives of higher order). So the expression for the rough Laplacian \eqref{roughlaplacian} in our context reduces to:
\begin{equation}\label{roughlaplacianridotto}
    \Delta \tau = -\, \left [\, \sum_{i=1}^p \, \nabla_{X_i}^Q \, \left ( \nabla_{X_i}^Q \, \tau \right ) \, + \sum_{j=1}^q \, \nabla_{Y_j}^Q \, \left ( \nabla_{Y_j}^Q \, \tau \right ) \,
 \right ] \,\, ,
 \end{equation}
 where, to simplify notation, we have written $X_i$ for $(X_i,0)$ and $Y_j$ for $(0,Y_j)$ . Using \eqref{derivatacovainQ}, \eqref{normaletoro} and \eqref{tensionfieldtorocalcolato} we find:
 \begin{eqnarray}\label{derivatacovaditauinQ}
   \nabla_{X_i}^Q \, \tau &=& \lambda \, \nabla_{X_i}^Q \, \eta_1^T \nonumber \\ \nonumber
    &=& \lambda \, \nabla_{X_i}^{\R^n} \, \eta_1^T +\, \frac{\lambda}{|\eta_1^Q|^2}\,\left [ \frac{1 }{c^2}\, \langle X_i,\eta_1^T\rangle  + \frac{1 }{d^2}\, \langle0,\eta_1^T\rangle  \right ] \, \eta_1^Q\\ \nonumber
    &=& \lambda \, \frac{c^2}{a^2} \, X_i + \, \frac{\lambda}{|\eta_1^Q|^2}\,[\,0\,]\, \, \eta_1^Q \\
    &=&\lambda \, \frac{c^2}{a^2} \, X_i \,\, .
 \end{eqnarray}
Next, using first \eqref{derivatacovaditauinQ},
\begin{eqnarray}\label{biderivatacovaditauinQ}
\nonumber
  \nabla_{X_i}^Q \,\left ( \nabla_{X_i}^Q \, \tau \right )&=& \lambda \, \frac{c^2}{a^2} \, \nabla_{X_i}^Q \,X_i\\ \nonumber
    &=& \lambda \, \frac{c^2}{a^2} \,\left [\, \nabla_{X_i}^T \, X_i+B^T\left( X_i,X_i\right)\, \right ]\\
    &=& -\,\lambda \, \frac{c^2}{a^2} \,\frac{1}{|\eta_1^T|^2}\, \left [\,\frac{c^2}{a^2}\,\langle X_i,X_i\rangle   \, \right ] \,\eta_1^T \,\,,
 \end{eqnarray}
where, in order to obtain the last equality, we have used the fact that our orthonormal frame is geodesic and also \eqref{secondaformaT}. Now, a very similar computation leads us to
\begin{equation}\label{biderivatacovaditauinQbis}
   \nabla_{Y_j}^Q \,\left ( \nabla_{Y_j}^Q \, \tau \right ) = -\,\lambda \, \frac{d^2}{b^2} \,\frac{1}{|\eta_1^T|^2}\, \left [\,\frac{d^2}{b^2}\,\langle Y_j,Y_j\rangle   \, \right ] \,\eta_1^T \,\,.
\end{equation}
Putting together \eqref{roughlaplacianridotto}, \eqref{biderivatacovaditauinQ} and \eqref{biderivatacovaditauinQbis} we obtain
\begin{equation}\label{deltatauconciso}
    \Delta \, \tau \, = \, \mu \,\,  \eta_1^T \,\,,
\end{equation}
where, taking into account \eqref{normacostantesutorobis}, we have defined the constant $\mu$ as follows:
\begin{equation}\label{definizionedimu}
    \mu = \frac{\lambda}{|\eta_1^T|^2} \, \left[\,\frac{p\,c^4}{a^4}+\frac{q\,d^4}{b^4}\, \right ]
\end{equation}
(note that, if \eqref{condizioneminimalita} does not hold, then $\lambda \neq 0$, so that $\mu \neq 0$ and the immersion is not minimal).

By way of summary, the previous computations have led us to the following conclusion:
\begin{equation}\label{bitensioneconcurvaturanonesplicita}
    \tau_2 \,=\, - \, \left [\,\mu \,\eta_1^T + {\trace}\, R^Q (d\,i, \tau)\,d\,i\,\, \right ] \,\,.
\end{equation}
We have to investigate for which values (if any) of $a,\,b$ the bitension $\tau_2$ vanishes. In order to deal in an efficient way with the curvature tensor, we shall study the vanishing of normal and tangential components separately. In particular, we shall prove that the normal component of $\tau_2$ is identically zero if and only if \eqref{condizioneproperbiharmonic} holds. The proof of the theorem will then be completed by the verification that the tangential part of $\tau_2$ vanishes for all values of $a$ and $b$. So, let us first study whether, for suitable values of $a$ and $b$, we can have
\begin{equation}\label{condizionedibitensionenormalenulla}
    \langle\, \tau_2, \, \eta_1^T \,\rangle   \,= \, 0\,\, .
\end{equation}
From \eqref{bitensioneconcurvaturanonesplicita} and \eqref{definizionedimu} we have:
$$
-\, \langle \tau_2, \, \eta_1^T \rangle   = \lambda \,
  \left[\,\left (\frac{p\,c^4}{a^4}+\frac{q\,d^4}{b^4}
  \right ) + \sum_{i=1}^p\, \langle R^Q(X_i, \tau)X_i,\eta_1^T\rangle  \,+ \sum_{i=1}^q\, \langle R^Q(Y_j, \tau)Y_j,\eta_1^T\rangle  \, \right ]\,\,.
$$
Next, we observe that
\begin{equation}\label{curvaturasezionalerimescolata}
    \langle R^Q(X_i, \tau)X_i,\eta_1^T\rangle  = -\, \frac{\langle R^Q(X_i, \tau)\eta_1^T,X_i\rangle  }{|\eta_1^T|^2} \, |\eta_1^T|^2=K^Q(X_i,\eta^T)\, |\eta_1^T|^2 \,\, ,
\end{equation}
where $K^Q(X_i,\eta^T)$ denotes sectional curvature, which (see \cite{DoCarmo}) can be expressed by means of:
\begin{equation}\label{curvaturasezionale}
    K^Q(X_i,\eta^T)=\langle B^Q(X_i,X_i),\,B^Q(\eta^T,\eta^T)\rangle  -\langle B^Q(X_i,\eta^T),
\,B^Q(X_i,\eta^T)\rangle   \,\, .
\end{equation}
By using \eqref{curvaturasezionalerimescolata} in \eqref{curvaturasezionale} and performing a computation which, according to \eqref{espressionefinaleBinQ}, uses
$$
B^Q(X_i,X_i)= -\, \frac{1}{|\eta_1^Q|^2}\, \frac{1}{c^2} \, \langle X_i,X_i\rangle  \, \eta_1^Q \,\, ,
$$
$$
B^Q(\eta_1^T,\eta_1^T)\rangle  =  -\, \frac{1}{|\eta_1^Q|^2}\, \left ( \frac{c^2}{a^2}+\frac{d^2}{b^2}\right ) \, \eta_1^Q \,\, ,
$$
\begin{equation} \label{utiledopopercomponentetangente}
B^Q(X_i,\eta^T)= 0 \,\, ,
\end{equation}
we find:
\begin{equation}\label{curvaturasezionaleformulaquasifinale}
    \langle R^Q(X_i, \eta_1^T)X_i,\eta_1^T\rangle  = -\, \frac{1}{|\eta_1^Q|^2}\, \frac{1}{c^2} \, \langle X_i,X_i\rangle  \, \left ( \frac{c^2}{a^2}+\frac{d^2}{b^2}\right ) \,\, .
\end{equation}
In a very similar fashion we also compute:
\begin{equation}\label{curvaturasezionaleformulaquasifinalebis}
    \langle R^Q(Y_j, \eta_1^T)Y_j,\eta_1^T\rangle  = -\, \frac{1}{|\eta_1^Q|^2}\, \frac{1}{d^2} \, \langle Y_j,Y_j\rangle  \, \left ( \frac{c^2}{a^2}+\frac{d^2}{b^2}\right ) \,\, .
\end{equation}
Putting together \eqref{curvaturasezionaleformulaquasifinale}, \eqref{curvaturasezionaleformulaquasifinalebis} and \eqref{bitensioneconcurvaturanonesplicita} it is easy to obtain the following conclusion:
\begin{equation}\label{magica}
 \langle  \tau_2, \, \eta_1^T \rangle   = -\, \lambda \, \left \{
  \left[\,\frac{p\,c^4}{a^4}+\frac{q\,d^4}{b^4}\, \right ] -\,\frac{1}{|\eta_1^Q|^2}\,\left[\,\frac{c^2}{a^2}+\frac{d^2}{b^2}\, \right ]\,\left[\,\frac{p}{c^2}+\frac{q}{d^2}\, \right ] \,\right \} \,\, .
\end{equation}
Now, using \eqref{normacostantesutoro} and \eqref{condizioneimmersionetoroinellissoide} in \eqref{magica}, it is not difficult to check that \eqref{condizionedibitensionenormalenulla} holds if and only if \eqref{condizioneproperbiharmonic} is satisfied.
At this stage, we can say that the proof of the theorem will be completed if we show that
\begin{equation}\label{condizionedibitensionetangentenulla}
    \langle \, \tau_2, \, W \,\rangle   \,= \, 0
\end{equation}
for any vector field $W$ which is tangent to the torus. Taking into account \eqref{bitensioneconcurvaturanonesplicita}, we see that \eqref{condizionedibitensionetangentenulla} is equivalent to:
\begin{equation*}\label{componentetangentenulla1}
    \langle \, {\trace}\, R^Q (d\,i, \tau)\,d\,i\,,\, W\,\rangle   =0 \,\, .
\end{equation*}
Because of \eqref{tensionfieldtorocalcolato}, it is enough to show that
\begin{equation*}\label{componentetangentenulla2}
    \langle \,  R^Q (X, \,\eta_1^T)\,X,\, W\,\rangle   =0
\end{equation*}
holds if $X,\,W$ are arbitrary vectors tangent to $T$. But, by the Gauss equation (see \cite{DoCarmo}), we deduce:
\begin{equation}\label{componentetangentenulla3}
    \langle \,  R^Q (X, \,\eta_1^T)\,X,\, W\,\rangle   =\langle \,B^Q(X,W),B^Q(\eta_1^T,X)\,\rangle   - \langle \,,B^Q(\eta_1^T,W),\,B^Q(X,X)\,\rangle   =0 \,\,,
\end{equation}
where, for the last equality, we have used \eqref{utiledopopercomponentetangente}.
\end{proof}

Next, we study biharmonic submanifolds into Euclidean ellipsoids of revolution $Q^{p+1}(c,d)$ defined as follows:
\begin{equation*}\label{def-ellissoide-rev}
    Q^{p+1}(c,d)=\left \{ (x,y)\in \, \R^{p+1} \times \R= \R^n \,\, : \,\, \frac{|x|^2}{c^2}+ \frac{y^2}{d^2} =1 \right \} \,\,,
\end{equation*}
where $c,\,d$ are fixed positive constants. In this case, the symmetry of $Q^{p+1}(c,d)$ makes it natural to look for biharmonic hyperspheres. More precisely, we shall study isometric immersions of the following type:
\begin{equation}\label{hyperspheres}
\left .
  \begin{array}{cccccc}
    i\,\,:\,\,& S^p(a) &\times& \{b \}& \longrightarrow&Q^{p+1}(c,d)\\
    &&&&& \\
    &(x_1,\ldots,x_{p+1}&,&b)&\longmapsto& (x_1,\ldots,x_{p+1},b\,\, ) \,\, , \\
  \end{array}
\right .
\end{equation}
where $i$ denotes the inclusion and the constants $a,b$ must again satisfy the condition
\begin{equation}\label{condizioneimmersioneipersferainellissoide}
    \frac{a^2}{c^2}+  \frac{b^2}{d^2}= 1
\end{equation}
(note that $a$ is a radius, so it is positive, while the only request on $b$ is: $|b| < d$).

In this context, we shall prove the following result:
\begin{theorem}\label{teorema2}Let $i\,:\,S^p(a) \times \{b \} \to Q^{p+1}(c,d)$ be an isometric immersion as in \eqref{hyperspheres}. If
\begin{equation}\label{condizioneminimalitaipersfera}
     a^2 = c^2 \,\, ; \quad b = 0 \, \,
\end{equation}
then the immersion is minimal (this is the case of the equator hypersphere). If
\begin{equation}\label{condizioneproperbiharmonicipersfera}
    a = c\, \,\sqrt { \frac{c}{c+d}} \,\, ; \quad b = \pm\, \,  d \,\, \sqrt { \frac{d}{c+d} }\,\, ,
\end{equation}
then the immersion is proper biharmonic.
\end{theorem}
\begin{proof} Again, we shall use a superscript $Q$ for objects concerning the ellipsoid, while the letter $S$ will appear for reference to the hypersphere $S=S^p(a) \times \{b\}$. Essentially, the proof follows the arguments of Theorem \ref{teorema1} and most of the calculations can be performed by setting $q=0$ in the formulas above: for this reason, we limit ourselves to point out the relevant differences only. First, let us assume that $b\neq0$ : normal vectors $\eta_1^Q$, $\eta^Q$, $\eta_1^S$ and $\eta^S$ can be introduced precisely as in \eqref{normaleunitariaellissoide}--\eqref{normacostantesutorobis}. We also note that, since $b\neq0$, \eqref{condiztangenzatoro} implies that a tangent vector to $S$ must be of the form
\begin{equation}\label{vettoretangenteipersfera}
    W=(X,0) \,\, .
\end{equation}
Taking into account \eqref{vettoretangenteipersfera} we easily obtain
\begin{equation*}\label{tensionfieldipersfera}
    \tau= \, \lambda \, \eta_1^S \,\, ,
\end{equation*}
where
\begin{equation*}\label{lambdaipersfera}
    \lambda= \,- \, \left [ \frac{c^4}{a^2}+\frac{d^4}{b^2} \right ]^{-1}\, \left [\, \frac{p \, c^2}{a^2}\, \right ]
\end{equation*}
(note that $\lambda\neq0$, so that in this case the hypersphere is not minimal).

In the computation of $\Delta \, \tau$ only the terms $\nabla_{X_i}^Q \, \left ( \nabla_{X_i}^Q \, \tau \right ) $ in \eqref{roughlaplacianridotto} are relevant: this fact leads us to the expression
\begin{equation}\label{roughlaplacianipersfera}
    \Delta \, \tau= \mu\, \eta_1^S \,\, ,
\end{equation}
where now
\begin{equation*}\label{muipersfera}
    \mu= \frac{\lambda}{|\eta_1^S|^2} \, \left[\,\frac{p\,c^4}{a^4}\, \right ] \,\, .
\end{equation*}
Also the calculation involving the curvature terms follows the lines above and leads us to
\begin{equation}\label{magicaipersfera}
    \langle  \tau_2, \, \eta_1^S \rangle   = -\, \lambda \, \left \{
  \left[\,\frac{p\,c^4}{a^4}\, \right ] -\,\frac{1}{|\eta_1^Q|^2}\,\left[\,\frac{c^2}{a^2}+\frac{d^2}{b^2}\, \right ]\,\left[\,\frac{p}{c^2}\, \right ] \,\right \} \,\, .
\end{equation}
Now, inspection of \eqref{magicaipersfera} shows that \eqref{condizioneproperbiharmonicipersfera} is equivalent to the vanishing of the normal component of the bitension. Finally, an argument as above shows that the tangential component of the bitension always vanishes, so ending the case $b\neq0$.

In the case that $b=0$ we observe that
\begin{equation}\label{normaleunitariaipersferab=0}
    \eta^S=(0,\,\ldots,\,1) \,\, .
\end{equation}
Using \eqref{normaleunitariaipersferab=0} in \eqref{secondaformaT} it is easy to conclude that, in this case, the second fundamental form of $S$ vanishes identically, so that the equator hypersphere is totally geodesic and so minimal, a fact which ends the theorem.
\end{proof}

\section{Composition properties}

Our first result is:
\begin{theorem}\label{composition-property-theorem1} Let $i: S^p(a)\to Q^{p+1}(c,d)$ be a proper biharmonic immersion as in Theorem \ref{teorema2}, and let $\varphi : M^m \to S^p(a)\,$be a minimal immersion. Then $i\circ \varphi : M^m\to Q^{p+1}(c,d)\,$is a proper biharmonic immersion.
\end{theorem}
\begin{proof} Let $W_i$ , $i=1,\, \ldots, m $ , be a local orthonormal frame on $M^m$. To simplify notation, for a tangent vector $W$ to $M^m$, we write $W$ for both $d \varphi (W)$ and $di \left (d \varphi (W)\right )$ . The composition law for the tension field (see \cite{EL83}), together with the minimality of $\varphi$, gives:
\begin{eqnarray}\label{composition-tensionfield}
  \tau (i\circ \varphi) &=& \sum_{i=1}^m \, \nabla d\,i \left ( W_i,W_i \right ) + \tau(\varphi)\nonumber \\ \nonumber
  &=& \sum_{i=1}^m \, \nabla d\,i \left ( W_i,W_i \right )\\
   &=& \sum_{i=1}^m \, B^S \left ( W_i,W_i \right ) \,\, .
\end{eqnarray}
Now, adapting the calculation of \eqref{secondaformaT}, we have:
\begin{equation}\label{secondaformaS}
    B^S \left ( W_i,W_i \right )= - \, \frac{c^2}{a^2}\, \frac{1}{|\eta_1^S|}\, \langle  W_i,W_i\rangle   \, \eta^S \,\, .
\end{equation}
Next, using \eqref{secondaformaS} in \eqref{composition-tensionfield}, we obtain
\begin{equation}\label{tension-composto}
   \tau (i\circ \varphi)= - \, m \,  \frac{c^2}{a^2}\, \frac{1}{|\eta_1^S|}\,  \eta^S \,\, .
\end{equation}
In particular, we deduce from \eqref{tension-composto} that $i\circ \varphi$ is not minimal and we proceed to the computation of the bitension. For convenience, we set
$$
\nu =  \, m\, \frac{c^2}{a^2}\, \frac{1}{|\eta_1^S|}\, \, .
$$
Using \eqref{tension-composto} we have:
\begin{eqnarray}\label{bitension-composto}
   \tau_2 (i\circ \varphi) &=& -\, \Delta^M \tau (i\circ \varphi)- \sum_{i=1}^m \,R^Q(W_i,\tau (i\circ \varphi) )W_i\nonumber \\
   &=&  \nu \, \left [ \,\Delta^M \eta^S +  \sum_{i=1}^m \,R^Q(W_i, \eta^S) )W_i \, \right ] \,\, .
\end{eqnarray}
Next, we study separately the two terms in the right-hand side of \eqref{bitension-composto}. First, computing as in \eqref{roughlaplacianipersfera} (with $p$ replaced by $m$), we find
\begin{equation}\label{roughlaplacianM}
    \Delta^M \eta^S= \left [ \, \frac{m}{|\eta_1^S|^2}\, \frac{c^4}{a^4}\right ] \,  \eta^S \, \, .
\end{equation}
Second, using the Gauss equation as in \eqref{componentetangentenulla3}, we obtain:
\begin{equation}\label{terminecurvaturaincompositionproperty}
\langle \,\sum_{i=1}^m \,R^Q(W_i, \eta^S) )W_i, \eta^S\,\rangle  =-\, m\, \frac{1}{|\eta_1^S|^2} \,\frac{1}{|\eta_1^Q|^2}\,\frac{1}{c^2}\, \left( \frac{c^2}{a^2}+\frac{d^2}{b^2}
 \right ) \,\, ,
 \end{equation}
 and
 \begin{equation}\label{terminecurvaturatangenteincompositionproperty}
\langle \,\sum_{i=1}^m \,R^Q(W_i, \eta^S) )W_i, W\,\rangle  =0
 \end{equation}
 for all vector $W$ which is tangent to $S$ . Putting together \eqref{bitension-composto}--\eqref{terminecurvaturatangenteincompositionproperty} we conclude that $\tau_2 (i\circ \varphi)$ is parallel to $\eta^S$ and vanishes if and only if
 \begin{equation}\label{altramagica}
  \left \{
  \left[\,\frac{c^4}{a^4}\, \right ] -\,\frac{1}{|\eta_1^Q|^2}\,\left[\,\frac{c^2}{a^2}+\frac{d^2}{b^2}\, \right ]\,\left[\,\frac{1}{c^2}\, \right ] \,\right \}=0 \,\, .
 \end{equation}
 But \eqref{altramagica} is equivalent to the two conditions \eqref{condizioneproperbiharmonicipersfera} and \eqref{condizioneimmersioneipersferainellissoide}, so that the proof is completed. \end{proof}

In a spirit similar to the previous theorem, we also obtain the following result:
\begin{theorem}\label{composition-property-theorem2} Let $i: S^p(a)\times S^q(b)\to Q^{p+q+1}(c,d)$ be a proper biharmonic immersion as in Theorem \ref{teorema1}, and let $\varphi_1 : M_1^{m_1} \to S^p(a)\,$, $\varphi_2 : M_2^{m_2} \to S^q(b)\,$be two minimal immersions. Then $i\circ (\varphi_1 \times \varphi_2 ): M_1^{m_1} \times M_2^{m_2} \to Q^{p+q+1}(c,d)\,$is a proper biharmonic immersion.
\end{theorem}
\begin{proof} The proof of this result is a straightforward variant of the arguments of Theorem \ref{composition-property-theorem1} and so the details are omitted.
\end{proof}

\begin{remark}
When $c=d=1$ the composition properties described in Theorem~\ref{composition-property-theorem1} and Theorem~\ref{composition-property-theorem2} reduce to those first proved in \cite{CMO02}. It is important
to note that all biharmonic sub\-mani\-folds into the ellipsoids constructed using the composition properties have parallel mean curvature vector field.
\end{remark}


%
%


\begin{thebibliography}{99}

\bibitem{BMO13} A. Balmu\c s, S.~Montaldo, C.~Oniciuc.
Biharmonic PNMC submanifolds in spheres. {\em Ark. Mat.} 51 (2013), 197--221.

\bibitem{BMO10} A. Balmu\c s, S.~Montaldo, C.~Oniciuc.
Biharmonic hypersurfaces in 4-dimensional space forms. {\em Math. Nachr.} 283 (2010), 1696--1705.

\bibitem{BMO08} A. Balmu\c s, S. Montaldo, C. Oniciuc.
Classification results for biharmonic submanifolds in spheres. {\em Israel J. Math.} 168 (2008), 201--220.

\bibitem{BO11} A. Balmu\c s, C. Oniciuc.
Biharmonic submanifolds with parallel mean
curvature vector field in spheres. {\em J. Math. Anal. Appl.} 386 (2012), 619--630.

\bibitem{BO09} A. Balmu\c s, C. Oniciuc.
Biharmonic surfaces of $\mathbb{S}^4$. {\em Kyushu J. Math.} 63 (2009), 339--345.

\bibitem{CMO02} R.~Caddeo, S.~Montaldo, C.~Oniciuc.
Biharmonic submanifolds in spheres. {\em Israel J. Math.} 130 (2002), 109--123.

\bibitem{CMO01} R.~Caddeo, S.~Montaldo, C.~Oniciuc.
 Biharmonic submanifolds of $\s^3$. {\it Internat. J. Math.} 12 (2001), 867--876.

%
%

\bibitem{Chen} B.-Y.~Chen, Some open problems and conjectures on submanifolds of finite type, {\em Soochow J. Math.} 17 (1991), 169--188.

\bibitem{DoCarmo} M.P.~Do Carmo. {\em Riemannian Geometry}. Birkh\"auser, 1992.

%
%
%
%
%
%
%
%
\bibitem{EL83} J.~Eells, L.~Lemaire. {\it  Selected topics in harmonic maps.} CBMS Regional Conference Series in Mathematics, 50. American Mathematical Society, Providence, RI, 1983.
%

\bibitem{Jiang} G.Y.~Jiang. {2-harmonic maps and their first and second variation formulas}.
{\it Chinese Ann. Math. Ser. A 7},  7 (1986), 130--144.



\bibitem{SMCO} S.~Montaldo, C.~Oniciuc. {A short survey on biharmonic maps between riemannian manifolds}. {\it Rev. Un. Mat. Argentina}, 47 (2006), 1--22.


\bibitem{Mont_Ratto} S.~Montaldo, A.~Ratto. Biharmonic curves into quadrics, arXiv:1309.0631.

\bibitem{OT} Y.-L.~Ou, L.~Tang. On the generalized Chen's conjecture on biharmonic submanifolds. {\em Michigan Math. J.}, 61 (2012), 531--542.

\bibitem{O02} C.~Oniciuc.
Biharmonic maps between Riemannian manifolds. An. Stiint. Univ. Al.I. Cuza Iasi Mat (N.S.) 48 (2002), 237--248.


\end{thebibliography}
\end{document}